\input ./style/arxiv-vmsta.cfg
\documentclass[numbers,compress,v1.0.1]{vmsta}
\usepackage{vtexbibtags}

\volume{5}
\issue{2}
\pubyear{2018}
\firstpage{207}
\lastpage{224}
\aid{VMSTA103}
\doi{10.15559/18-VMSTA103}
\articletype{research-article}

\startlocaldefs
\newcommand{\rrVert}{\Vert}
\newcommand{\llVert}{\Vert}
\allowdisplaybreaks

\newcommand{\eL}{{\cal L}}
\newcommand{\w}{\widehat}
\newcommand{\qubar}{\overline{q}}
\newcommand{\pbar}{\overline{p}}
\newcommand{\Expect}{\mathrm{E}} 
\newcommand{\Var}{\mathrm{Var}} 
\newcommand{\Prob}{\mathrm{P}}
\newcommand{\ii}{{\mathrm i}}
\newcommand{\ee}{{\mathrm e}}
\newcommand{\dd}{{\mathrm{d}}}
\newcommand{\integrald}{{\mathrm{d}}} %
\newcommand{\dirac}{I} 
\newcommand{\NB}{\mathrm{NB}} 
\newcommand{\RR}{\mathbb{R}}
\newcommand{\NN}{\mathbb{N}}
\newcommand{\ZZ}{\mathbb{Z}}
\newcommand{\F}{{\mathcal{F}}}
\newcommand{\M}{{\mathcal{M}}}
\newtheorem{theorem}{Theorem}[section]
\newtheorem{corollary}{Corollary}[section]
\newtheorem{lemma}{Lemma}[section]
\newtheorem{remark}{Remark}[section]
\endlocaldefs

\begin{document}
\begin{frontmatter}
\pretitle{Research Article}

\title{On closeness of two discrete weighted sums}

\author[a]{\inits{V.}\fnms{Vydas}~\snm{\v{C}ekanavi\v{c}ius}\thanksref{cor1}\ead[label=e1]{vydas.cekanavicius@smif.vu.lt}}
\thankstext[type=corresp,id=cor1]{Corresponding author.}
\author[b]{\inits{P.}\fnms{Palaniappan}~\snm{Vellaisamy}\ead[label=e2]{pv@math.iitb.ac.in}}
\address[a]{Department of Mathematics and Informatics,
\institution{Vilnius University},\break
Naugarduko 24, Vilnius 03225, \cny{Lithuania}}
\address[b]{Department of Mathematics,
\institution{Indian Institute of
Technology Bombay},\break
Powai, Mumbai-400076, \cny{India}}



\markboth{V. \v{C}ekanavi\v{c}ius and P. Vellaisamy}{On closeness of
two discrete weighted sums}

\begin{abstract}
The effect that weighted summands have on each other
in approximations of $S=w_1S_1+w_2S_2+\cdots+w_NS_N$ is investigated.
Here, $S_i$'s are sums
of integer-valued random variables, and $w_i$ denote weights,
$i=1,\dots,N$. Two cases are considered: the general case of
independent random variables when their closeness
is ensured by the matching of factorial moments and the case when the
$S_i$ has the Markov Binomial distribution.
The Kolmogorov metric is used to estimate the accuracy of
approximation.
\end{abstract}

\received{\sday{1} \smonth{2} \syear{2018}}
\revised{\sday{16} \smonth{4} \syear{2018}}
\accepted{\sday{27} \smonth{4} \syear{2018}}
\publishedonline{\sday{21} \smonth{5} \syear{2018}}
\end{frontmatter}

\section{Introduction}\label{sec1}

Let us consider a typical cluster sampling design: the
entire population consists of different clusters, and the
probability for each cluster to be selected into a sample is
known. The sum of sample elements is then equal to
$S=w_1S_1+w_2S_2+\cdots+w_NS_N$. Here, $S_i$ is the sum of
independent identically distributed (iid) random variables (rvs) from
the $i$-th cluster.
A similar situation arises in actuarial mathematics
when the sum $S$ models the discounted amount of the total net loss of
a company, see, for example,
\cite{TaTs03}. Note that then $S_i$ may be the sum of dependent rvs. Of
course, in actuarial models,
$w_i$ are also typically random, which makes our research just a first
step in this direction.
In many papers, the limiting behavior of weighted sums is
investigated with the emphasis on weights or tails of distributions,
see, for example, \cite
{DLS17,LChS17,Liang06,MOC12,Soo01,YCS11,YLS16,YW12,W11,WV16,ZSW09},
and references therein. We, 
however,
concentrate on the impact of $S-w_iS_i$ on $w_iS_i$.
Our research is motivated by the following simple example. Let us
assume that $S_i$ is in some sense close to $Z_i$, $i=1,2$. Then a
natural approximation to $w_1S_1+w_2S_2$ is $w_1Z_1+w_2Z_2$. Suppose
that we want to estimate the closeness of both sums in some metric
$d(\cdot,\cdot)$. The standard approach which works for the majority of
metrics then gives
\begin{equation}
d(w_1S_1+w_2S_2,w_1Z_1+w_2Z_2)
\leqslant d(w_1S_1,w_1Z_1)+d(w_2S_2,w_2Z_2).\label{afirst}
\end{equation}
The triangle inequality (\ref{afirst}) is not always useful. For example,
let $S_1$ and $Z_1$ have the same Poisson distribution with parameter
$n$ and let $S_2$ and $Z_2$ be Bernoulli variables with probabilities
1/3 and 1/4, respectively. Then (\ref{afirst}) ensures the trivial order of
approximation $O(1)$ only. Meanwhile, both $S$ and $Z$ can be treated
as small (albeit different) perturbations to the same Poisson variable
and, therefore, one can expect closeness of their distributions at
least for large $n$. The `smoothing' effect that other sums have on the
approximation of $w_iS_i$ is already observed in \cite{ElCe15} (see
also references therein). For some general results involving the
concentration functions, see, for example, \cite{Hipp85,Ro03}.

To make our goals more explicit, we need additional notation. Let $\ZZ$
denote the set of all integers. Let
$\F$ (resp.~$\F_Z$, resp.~$\M$) denote the set of
probability distributions (resp. distributions concentrated on integers,
resp.~finite signed measures) on $\RR$.
Let $I_a$ denote the distribution concentrated at real $a$ and set $I =
I_0$ . Henceforth,
the products and powers of measures are understood in the convolution sense.
Further, for a measure $M$, we set $M^0 = I$ and $\exp\{M\}=\sum_{k=0}^\infty M^k/k!$.
We denote by $\widehat M(t)$ the Fourier--Stieltjes
transform of $M$. The real part of $\w M(t)$ is denoted by $Re\w
M(t)$. Observe also that $\w{\exp\{M(t)\}}=\exp\{\w M(t)\}$. We
also use $\eL(\xi)$ to denote the distribution of $\xi$.

The
Kolmogorov (uniform) norm $\vert M\vert_K$ and the total variation
norm $\| M\|$ of $M$ are defined by
\[
\vert M\vert_K=\sup_{x\in\RR}\bigl\vert M\bigl((-\infty,x]\bigr)
\bigr\vert,\qquad \| M\|=M^{+}\{\RR\}+M^{-}\{\RR\},
\]
respectively. Here $M=M^{+}-M^{-}$ is the Jordan--Hahn
decomposition of $M$.
Also, for any two measures $M$ and $V$, $\vert M\vert_K\leqslant\| M\|$,
$\vert MV\vert_K\leqslant\| M\|\cdot\vert V\vert_K$, $\vert\w
M(t)\vert\leqslant\| M\|$, $\| \exp\{M\}\|\leqslant\exp\{\| M\|\}
$. If $F\in\F$, then
$\vert F\vert_K=\| F\|=\| \exp\{F-\dirac\}\|=1$. Observe also that, if
$M$ is concentrated on integers, then
\[
M=\sum_{k=-\infty}^\infty M\{k\}\,
\dirac_k,\qquad\w M(t)=\sum_{k=-\infty
}^\infty
\ee^{\ii tk}M\{k\},\qquad\| M\|=\sum_{k=-\infty}^\infty
\bigl\vert M\{k\}\bigr\vert.
\]

For $F\in\F$, $h\geqslant0$, L\' evy's
concentration function is defined by
\begin{equation*}
Q(F,h)=\sup_xF \bigl\{[x,x+h] \bigr\}.
\end{equation*}

All absolute positive
constants are denoted by the same symbol $C$. Sometimes to avoid
possible ambiguities, the constants $C$ are supplied
with indices. Also, the constants depending on parameter $N$ are
denoted by $C(N)$. We also assume usual conventions $\sum_{j=a}^b=0$
and $\prod_{j=a}^b=1$, if $b<a$. The notation $\varTheta$ is used for any
signed measure satisfying $\| \varTheta\|\leqslant1$. The notation
$\theta$ is used for any real or complex number satisfying $\vert
\theta\vert\leqslant1$.

\section{Sums of independent rvs}

The results of this section are partially inspired by a comprehensive
analytic research of probability generating functions in \cite{Hw99}
and the papers on \emph{mod}-Poisson convergence, see \cite
{BKN14,KN10,KNN15}, and references therein.
Assumptions in the above-mentioned papers are made about the behavior
of characteristic or probability generating functions. The inversion
inequalities are then used to translate their differences 
to the differences of distributions.
In principle, \emph{mod}-Poisson convergence means that if an initial
rv is a perturbation of some Poisson
rv, then their distributions must be close. Formally, it is required
for $\exp\{-\tilde\lambda_n(\ee^{\ii t}-1)\}f_n(t)$ to have a limit
for some sequence of Poisson parameters $\tilde\lambda_n$, as $n\to
\infty$. Here, $f_n(t)$ is a characteristic function of an investigated
rv. Division by a certain Poisson characteristic function is one of the
crucial steps in the proof of Theorem \ref{Teorema} below, which makes it
applicable to rvs satisfying the \emph{mod}-Poisson convergence
definition, provided they can be expressed as sums of independent rvs.
Though we use factorial moments, similar to Section 7.1 in \cite
{BKN14}, our work is much more closer in spirit to \cite{SC88},
where general lemmas about the closeness of lattice measures are
proved. 

In this section, we consider a general case of independent
non-identically distributed rvs, forming a triangular array (a scheme
of series). Let $S_i=X_{i1}+X_{i2}+\cdots+X_{in_i}$,
$Z_i=Z_{i1}+Z_{i2}+\cdots+Z_{in_i}$, $i=1,2,\dots,N$. We assume that
all the $X_{ij}$, $Z_{ij}$ are mutually independent and integer-valued.
Observe that, in general, $S=\sum_{i=1}^Nw_iS_i$ and $Z=\sum_{i=1}^Nw_iZ_i$ are
\emph{not} integer-valued and, therefore, the
standard methods of estimation of lattice rvs do not apply. Note also
that, since any infinitely divisible distribution can be expressed as a
sum of rvs, Poisson, compound Poisson and negative binomial rvs can be
used as $Z_i$.

The distribution of $X_{ij}$ (resp. $Z_{ij}$) is denoted by $F_{ij}$
(resp. $G_{ij}$). The closeness of characteristic functions will be
determined by the closeness of corresponding factorial moments. Though
it is proposed in \cite{BKN14} to use standard factorial moments even
for rvs taking negative values, we think that right-hand side and
left-hand side factorial moments, already used in \cite{SC88}, are more
natural characteristics.
Let, for $k=1,2,\dots$, and any $F\in\F_Z$,
\begin{align*}
\nu_k^{+}(F_{ij})&=\sum
_{m=k}^\infty m(m-1)\cdots(m-k+1)F_{ij}\{m\},
\\
\nu_k^{-}(F_{ij})&=\sum
_{m=k}^\infty m(m-1)\cdots(m-k+1)F_{ij}\{-m\}.
\end{align*}
For the estimation of the remainder terms we also need the following
notation:\break
$\beta_k^{\pm}(F_{ij},G_{ij})=\nu_k^{\pm}(F_{ij})+\nu_k^{\pm
}(G_{ij})$, $\sigma^2_{ij}=\max(\Var(X_{ij}),\Var(Z_{ij}))$, and\vadjust{\goodbreak}
\begin{align*}
u_{ij}&=\min \biggl\{1-\frac{1}{2}\bigl\| F_{ij}(
\dirac_1-\dirac)\bigr\| ;1-\frac
{1}{2}\bigl\| G_{ij}(
\dirac_1-\dirac)\bigr\| \biggr\}
\\
&=\min \Biggl\{\sum_{k=-\infty}^\infty\min
\bigl(F_{ij}\{k\},F_{ij}\{k-1\} \bigr);\sum
_{k=-\infty}^\infty\min \bigl(G_{ij}\{k
\},G_{ij}\{k-1\} \bigr) \Biggr\}.
\end{align*}
For the last equality, see (1.9) and (5.15) in \cite{Ce16}. Next we
formulate our assumptions. For some fixed integer $s\geqslant1$,
$i=1,\dots, N,\ j=1,\dots,n_i$,
\begin{align}
u_{ij}&>0,\qquad\sum_{j=1}^{n_i}u_{ij}
\geqslant1, \qquad n_i\geqslant1,\qquad w_i>0,\label{sal1}
\\
\nu_k^{+}(F_{ij})&=\nu_k^{+}(G_{ij}),
\qquad\nu_k^{-}(F_{ij})=\nu_k^{-}(G_{ij}),
\quad k=1,2,\dots,s\label{sal2}
\\
\beta_{s+1}^{+}(F_{ij},G_{ij})&+
\beta_{s+1}^{-}(F_{ij},G_{ij})<\infty.
\label{sal4} 
\end{align}

Now we are in position to formulate the main result of this section.


%
\begin{theorem} \label{Teorema} Let assumptions (\ref{sal1})--(\ref
{sal4}) hold. Then
\begin{align}
\!\!\!\bigl\vert\eL(S)-\eL(Z)\bigr\vert_K &\leqslant C(N,s)\frac{\max_jw_j}{\min_j
w_j}
\Biggl(\sum_{i=1}^N\sum
_{l=1}^{n_i}u_{il} \Biggr)^{-1/2}
\prod_{l=1}^N \Biggl(1+\sum
_{k=1}^{n_l}\sigma^2_{lk} /\sum
_{k=1}^{n_l} u_{lk} \Biggr)
\nonumber
\\
&\quad\times \sum_{i=1}^N\sum
_{j=1}^{n_i} \bigl[\beta^{+}_{s+1}(F_{ij},G_{ij})+
\beta_{s+1}^{-}(F_{ij},G_{ij})\bigr]
\Biggl(\sum_{k=1}^{n_i}u_{ik}
\Biggr)^{-s/2}. \label{BTa}
\end{align}
If, in addition, $s$ is even and
$\beta_{s+2}^{+}(F_{ij},G_{ij})+\beta_{s+2}^{-}(F_{ij},G_{ij})<\infty$,
then
\begin{align}
\bigl\vert\eL(S)-\eL(Z)\bigr\vert_K &\leqslant C(N,s)\frac{\max_jw_j}{\min_j
w_j}
\Biggl(\sum_{i=1}^N\sum
_{l=1}^{n_i}u_{il} \Biggr)^{-1/2}
\prod_{l=1}^N \Biggl(1+\sum
_{k=1}^{n_l}\sigma^2_{lk} /\sum
_{k=1}^{n_l} u_{lk} \Biggr)
\nonumber
\\
&\quad\times \sum_{i=1}^N\sum
_{j=1}^{n_i} \Biggl(\sum_{k=1}^{n_i}u_{ik}
\Biggr)^{-s/2} \Biggl(\bigl\vert\beta ^{+}_{s+1}(F_{ij},G_{ij})-
\beta_{s+1}^{-}(F_{ij},G_{ij})\bigr\vert
\nonumber
\\
&\quad+ \bigl[\beta^{+}_{s+2}(F_{ij},G_{ij})+
\beta _{s+2}^{-}(F_{ij},G_{ij})
\nonumber
\\
&\quad+ \beta^{-}_{s+1}(F_{ij},G_{ij})
\bigr] \Biggl(\sum_{k=1}^{n_i}u_{ik}
\Biggr)^{-1/2} \Biggr). \label{BTb}
\end{align}
\end{theorem}

The factor $(\sum_{i=1}^n\sum_{j=1}^{n_i}u_{ij})^{-1/2}$ estimates the
impact of $S$ on approximation of $w_iS_i$.
The estimate (\ref{BTb}) takes care of a possible symmetry of distributions.


If, in each sum $S_i$ and $Z_i$, all the rvs are identically
distributed, then we can get rid of the factor containing variances.
We say that condition (\textit{ID}) is satisfied if, for each $i=1,2,\dots,N$,
all rvs $X_{ij}$ and $Z_{ij}$ ($j=1,\dots, n_i$) are iid with
distributions $F_i$
and $G_i$, respectively. Observe, that if condition (\textit{ID}) is
satisfied, then the characteristic functions of $S$ and $Z$ are
respectively equal to
\[
\prod_{i=1}^N\w F_i^{n_i}(w_it),
\qquad\prod_{i=1}^N\w G_i^{n_i}(w_it).
\]
We also use notation $u_i$ instead of $u_{ij}$, since now
$u_{i1}=u_{i2}=\cdots=u_{in_i}$.\vadjust{\goodbreak}




\begin{theorem}\label{TeoremaID} Let the assumptions (\ref
{sal1})--(\ref
{sal4}) and the condition (ID) hold. Then
\begin{align}
\bigl\vert\eL(S)-\eL(Z)\bigr\vert_K&\leqslant C(N,s)\frac{\max_jw_j}{\min_j
w_j} \Biggl(
\sum_{i=1}^Nn_iu_{i}
\Biggr)^{-1/2}
\nonumber
\\
&\quad\times\sum_{i=1}^N
\frac{\beta^{+}_{s+1}(F_{i},G_{i})+\beta
_{s+1}^{-}(F_{i},G_{i})}{n_i^{s/2-1}u_i^{s/2}}. \label{BTcc}
\end{align}
\end{theorem}


How does Theorem \ref{Teorema} compare to the known results? In \cite
{CeEl14}, compound Poisson-type approximations to non-negative iid rvs
in each sum were considered under the additional Franken-type condition:
\begin{equation}
\nu_1^{+}(F_j)- \bigl(\nu_1^{+}(F_j)
\bigr)^2-\nu_2^{+}(F_j)>0,
\label{Franken}
\end{equation}
see \cite{Frank64}.
Similar assumptions were used in \cite{ElCe15,SC88}. Observe that
Franken's condition requires almost all probabilistic mass to be
concentrated at 0 and 1. Indeed, then
$\nu_1^{+}(F_j)<1$ and $F_j\{1\}\geqslant\sum_{k=3}^\infty
k(k-2)F_j\{
k\}$.
Meanwhile, Theorems \ref{Teorema} and \ref{TeoremaID} hold under much
milder assumptions and, as demonstrated in the example below, can be
useful even if (\ref{Franken}) is not satisfied. Therefore, even for
the case of one sum when $N=1$, our results are new.


\textbf{Example.} Let $N=2$, $w_1=1$, $w_2=\sqrt{2}$, and $F_j$ and
$G_j$ be defined by
$F_j\{0\}=0.375$, $F_j\{1\}=0.5$, $F_j\{4\}=0.125$,
$G_j\{0\}=0.45$, $G_j\{1\}=0.25$, $G_j\{2\}=0.25$,
$G_j\{5\}=0.05$, $(j=1,2)$. We assume that $n_2=n$ and $n_1=\lceil
\sqrt
n\,\rceil$ is the smallest integer greater or equal to $\sqrt{n}$. Then
$\nu_k^{+}(F_j)=\nu_k^{+}(G_j)$, $k=1,2,3$, $\beta_4^{+}(F_j,G_j)=9$,
$u_j=3/8$, $(j=1,2)$.
Therefore, by Theorem \ref{TeoremaID}
\[
\bigl\vert\eL(S)-\eL(Z)\bigr\vert_K\leqslant\frac{C}{\sqrt
{n_1+n_2}} \biggl(
\frac
{1}{n_1}+\frac{1}{n_2} \biggr)=O \bigl(n^{-1} \bigr).
\]
In this case, Franken's condition (\ref{Franken}) is
not satisfied, since
$\nu_1^{+}(F_j)-\nu_2^{+}(F_j)-\break(\nu_1^{+}(F_j))^2<0$.


Next we apply Theorem \ref{TeoremaID} to the negative binomial
distribution. For real $r > 0$ and $0<\tilde p <1$, let $\xi\sim\NB
(r, \tilde p)$
denote the distribution with
\begin{equation*}
\label{defY} \Prob(\xi= k) = {r + k -1 \choose k} \tilde p^r
\tilde q^k, \quad k =0,1,\ldots.
\end{equation*}
Here $\tilde q = 1 -\tilde p$. Note that $r$ is not
necessarily an integer.

Let $X_{1j}$ be concentrated on non-negative
integers ($\nu_k^{-}(F_j)=0$). We approximate $S_i$ by
$Z_i\sim\NB(r_i,p_i)$ with
\begin{equation*}
r_i = \frac{(\Expect S_i)^2}{\Var S_i - \Expect S_i}, \qquad\tilde p_i =
\frac{\Expect S_i}{\Var S_i},\label{alpu}
\end{equation*}
so that $\Expect S_i = r_i \tilde q_i/\tilde p_i$ and
$\Var S_i = r_i \tilde q_i/\tilde p_i^2$.
Observe that
%
\begin{equation}
\label{nbg} 
\w G_j(t)= \biggl(\frac{\tilde p_j}{1-\tilde q_j\ee^{\ii
t}}
\biggr)^{r_j/n_j}.
\end{equation}


\begin{corollary} \label{cornb} Let assumptions of Theorem \ref
{TeoremaID} hold with $X_{1j}$ concentrated on non-negative
integers and let $\Expect X^3_{1j}<\infty$, $(j=1,\dots,N)$. Let $G_j$
be defined by (\ref{nbg}). Then
\begin{align}
\bigl\vert\eL(S)-\eL(Z)\bigr\vert_K &\leqslant C\frac{\max_jw_j}{\min_j
w_j} \Biggl(\sum
_{i=1}^Nn_i\tilde
u_i \Biggr)^{-1/2}
\nonumber
\\
&\quad\times\sum_{k=1}^N \biggl[
\nu_3^{+}(F_k)+\nu_1^{+}(F_k)
\nu_2^{+}(F_k)+ \bigl(\nu_1^{+}(F_k)
\bigr)^3
\nonumber
\\
&\quad+\frac{(\nu_2^{+}(F_k)-(\nu_1^{+}(F_k))^2)^2}{\nu
_1^{+}(F_k)} \biggr]\tilde u_k^{-1}.
\label{nba}
\end{align}
Here
\[
\tilde u_k=1-\frac{1}{2}\max \biggl(\bigl\| ( \dirac_1-
\dirac)F_k\bigr\|, \biggl(r_k\ln \frac{1}{\tilde p_k}
\biggr)^{-1/2} \biggr).
\]
\end{corollary}


\begin{remark} (i) Note that
\[
r_k\ln\frac{1}{\tilde p_k}=\frac{(\nu_1^{+}(F_k))^2}{\nu_2^{+}(F_k)-
(\nu_1^{+}(F_k))^2}\ln\frac{\nu_2^{+}(F_k)-(\nu_1^{+}(F_k))^2+\nu
_1^{+}(F_k)}{\nu_1^{+}(F_k)}.
\]

(ii) Let $\nu_k^{+}(F_j)\asymp C, w_j\asymp C$. Then the accuracy of
approximation in (\ref{nba}) is of the order $O((n_1+\cdots+n_N)^{-1/2})$.
\end{remark}


\section{Sums of Markov Binomial rvs}

We already mentioned that it is not always natural to assume
independence of rvs. In this section, we still assume that
$S=w_1S_1+w_2S_2+\cdots+w_NS_N$ with mutually independent $S_i$. On the
other hand, we assume that each $S_i$ has a Markov Binomial (MB)
distribution, that is, $S_i$ is
a sum of Markov dependent Bernoulli variables.
Such a sum $S$ has a slightly more realistic interpretation in
actuarial mathematics. Assume, for example, that we have $N$ insurance
policy holders, $i$-th of whom can get ill during an insurance period
and be paid a claim $w_i$. The health of the policy holder depends on the
state of her/his health in the previous period. Therefore, we have a
natural two state (healthy, ill) Markov chain. Then $S_i$ is an
aggregate claim for $i$th insurance policy holder after $n_i$ periods,
meanwhile $S$ is an aggregate claim of all holders.
Limit behavior of the MB distribution is a popular topic among
mathematicians, discussed in numerous papers, see, for example, \cite
{CeVe10,Gan82,Aki93}, and references therein.

Let $0,\xi_{i1},\dots,\xi_{in_i},\dots$ , ($i=1,2,\dots,N$) be a Markov
chain with the transition probabilities
\begin{align*}
&\Prob(\xi_{ik}=1\,|\,\xi_{i,k-1}=1)=p_i, \qquad
\Prob(\xi_{ik}=0\,|\,\xi_{i,k-1}=1)=q_i,
\\
&\Prob(\xi_{i,k}=1\,|\,\xi_{i,k-1}=0)= \qubar_i,
\qquad\Prob(\xi_{ik}=0\,|\,\xi_{i,k-1}=0)= \pbar_i,
\\
& p_i+q_i=\qubar_i+ \pbar_i=1,
\quad\qquad p_i,\qubar_i\in (0,1),\quad k\in\NN.
\end{align*}
The distribution of $S_i=\xi_{i1}+\cdots+\xi_{in_i}$
$(n_i\in\NN)$ is called
the Markov binomial distribution with parameters $p_i,q_i,\pbar
_i,\qubar
_i,n_i$. The definition of a MB rv slightly differs from paper to
paper. We use the one from \cite{CeVe10}. Note that the Markov chain,
considered above, is not necessarily stationary.
Furthermore, the distribution of $w_iS_i$ is denoted by $H_{in}=\eL
(w_iS_i)$. For approximation of $H_{in}$ we use the signed compound
Poisson (CP) measure with matching mean and variance. Such signed CP
approximations usually outperform both the normal and CP
approximations, see, for example,
\cite{BaXi99,CeVe10,Ro03}.
Let
\[
\gamma_i=\frac{q_i\qubar_i}{q_i+\qubar_i},\qquad\w Y_i(t)=
\frac{q_i\ee
^{\ii w_it}}{1-p_i\ee^{\ii w_i t}}-1.
\]
Observe that $\w Y_i(t)+1$ is the characteristic function of the
geometric distribution. Let $Y_i$ be a measure corresponding to $\w Y_i(t)$.
For approximation of $H_{in}$ we use the signed CP measure $D_{in}$
%
\begin{align}
D_{in}&=\exp \biggl\{ \biggl(\frac{\gamma_i(\qubar
_i-p_i)}{q_i+\qubar
_i}+n_i
\gamma_i \biggr) Y_i
\nonumber
\\
&\quad-n_i \biggl(\frac{q_i\qubar_i^2}{(q_i+\qubar_i)^2} \biggl(p_i+
\frac
{q_i}{q_i+\qubar_i} \biggr)+\frac{\gamma_i^2}{2} \biggr) Y_i^2
\biggr\}. \label{Din}
\end{align}
The CP limit occurs when $n\qubar_i\to\tilde\lambda$, see, for
example, \cite{CeVe10}. Therefore, we assume $\qubar_i$ to be small,
though not necessarily vanishing. Let, for some fixed integer
$k_0\geqslant2$,
\begin{equation}
\qubar_i\,{\geqslant}\,\frac{1}{n^{k_0}},\qquad0\,{<}
\,p_i \,{\leqslant}\,\frac{1}{2},\qquad\qubar_i\,{
\leqslant}\,\frac{1}{30}, \qquad w_i\,{>}\,0,\qquad
n_i\,{\geqslant}\,1, \quad i\,{=}\,1,\dots,N. \label{cond1}
\end{equation}
In principle, the first assumption in (\ref{cond1}) can be dropped, but
then exponentially vanishing remainder terms appear in all results,
making them very complicated.

\begin{theorem} \label{TMARK} Let $H_{in}=\eL(w_iS_i)$ and let $D_{in}$
be defined by (\ref{Din}), $i=1,\dots,N$. Let the conditions stated in
(\ref{cond1}) be satisfied. Then
\begin{equation}
\label{MBT} \Biggl\vert\prod_{i=1}^NH_{in}-
\prod_{i=1}^ND_{in}
\Biggr\vert_K \leqslant C(N,k_0)\frac
{\max w_i}{\min w_i}\cdot
\frac{\sum_{i=1}^N\qubar_i(p_i+\qubar
_i)}{\sqrt{
\sum_{k=1}^N\max(n_k\qubar_k,1)}}.
\end{equation}
\end{theorem}

\begin{remark}
Let all $\qubar_i\geqslant C$, $i=1,\dots,N$. Then, obviously, the
right-hand side of (\ref{MBT}) is majorized by
\[
C(N,k_0)\frac{\max w_i}{\min w_i}\cdot\frac{1}{\sqrt{
\max n_k}}.
\]
Therefore, even in this case, the result is comparable with the
Berry--Esseen theorem.
\end{remark}
%
\section{Auxiliary results}

\begin{lemma}\label{ad} Let $h>0$, $W\in\M$, $W\{\RR\}=0$, $U\in\F
$ and
$\vert\widehat U(t)\vert\leqslant C\widehat V(t)$, for
$\vert t\vert\leqslant1/h$ and some symmetric distribution $V$ having
non-negative characteristic function. Then
\begin{align*}
\vert W U\vert_K&\leqslant C\int_{\vert t\vert\leqslant1/h} \biggl\vert
\frac{\widehat W(t)\widehat U(t)}{t} \biggr\vert\,\dd t + C \llVert W \rrVert Q(U, h)
\nonumber
\\
&\leqslant C \biggl(\sup_{\vert t\vert\leqslant1/h} \frac{\vert
\widehat W(t)\vert}{\vert t\vert}\cdot
\frac{1}{h}+ \llVert W \rrVert \biggr)Q(V, h). 
\end{align*}
\end{lemma}

Lemma \ref{ad} is a version of Le Cam's smoothing inequality, see
Lemma 9.3 in \cite{Ce16} and Lemma 3 on p. 402 in \cite{LeCam86}.

%
\begin{lemma}\label{ac} Let $F \in\F$, $h>0$ and $a>0$. Then
\begin{align}
Q(F, h)&\leqslant \biggl(\frac{96}{95} \biggr)^2h\int
_{\vert t\vert\leqslant1/h} \bigl\vert\widehat F(t) \bigr\vert\, \integrald t, \label{ac1}
\\
Q(F, h)&\leqslant \biggl(1+ \biggl(\frac{h}{a} \biggr) \biggr)
Q(F, a), \label{ac3}
\\
Q \bigl(\exp\bigl\{a(F-I)\bigr\}, h \bigr)&\leqslant\frac{C}{\sqrt{aF \{\vert
x\vert>h \} }}.
\label{ac4}
\end{align}
If, in addition, $\widehat F(t)\geqslant0$, then
\begin{equation}
h \int_{\vert t\vert\leqslant1/h} \bigl\vert\widehat F(t)\bigr\vert\, \integrald t \leqslant
CQ(F, h). \label{ac5}
\end{equation}
\end{lemma}
Lemma \ref{ac} contains well-known properties of Levy's
concentration function, see, for example, Chapter 1 in \cite{Pet95} or
Section 1.5 in \cite{Ce16}.

Expansion in left-hand and right-hand factorial moments for
Fourier--Stieltjes transforms is given
in \cite{SC88}. Here we need its analogue for distributions.
\begin{lemma}\label{bee} Let $F\in\F_Z$ and, for some $s\geqslant1$,
$\nu_{s+1}^{+}(F)+\nu_{s+1}^{-}(F)<\infty$. Then
\begin{align}
F&=\dirac+\sum_{m=1}^s
\frac{\nu_m^{+}(F)}{m!}(\dirac_1-\dirac)^m+\sum
_{m=1}^s\frac{\nu_m^{-}(F)}{m!}(\dirac_{-1}-
\dirac)^m
\nonumber
\\
&\quad+ \frac{\nu_{s+1}^{+}(F)+\nu_{s+1}^{-}(F)}{(s+1)!}(\dirac _1-\dirac)^{s+1}
\varTheta. \label{nuF}
\end{align}
\end{lemma}

\begin{proof}
For measures, concentrated on non-negative integers, (\ref{nuF}) is given
in \cite{Ce16}, Lemma 2.1. Observe that distribution $F$ can be
expressed as a mixture $F=p^{+}F^{+}+p^{-}F^{-}$ of distributions
$F^{+}$, $F^{-}$ concentrated on non-negative and negative integers,
respectively. Then Lemma 2.1 from \cite{Ce16} can be applied in turn to
$F^{+}$ and to $F^{-}$ (with $\dirac_{-1}$). The remainder terms can be
combined, since $(\dirac_{-1}-\dirac)=\dirac_{-1}(\dirac-\dirac
_1)=(\dirac_1-\dirac)\varTheta$.
\end{proof}

\begin{lemma}\label{f-g} Let $F,G\in\F_Z$ and, for some $s\geqslant1$,
$\nu_j^{+}(F)=\nu_{j}^{+}(G)$, $\nu_j^{-}(F)=\nu_j^{-}(G)$,
$(j=1,2,\dots,s)$. If
$\beta_{s+1}^{+}(F,G)+\beta_{s+1}(F,G)<\infty$, then
\[
F-G=\frac{\beta_{s+1}^{+}(F,G)+\beta_{s+1}^{-}(F,G)}{(s+1)!}(\dirac _1-\dirac)^{s+1}\varTheta.
\]
If, in addition, $\beta_{s+2}^{+}(F,G)+\beta_{s+2}(F,G)<\infty$ and
$s$ is even, then
\begin{align*}
F-G&=\frac{\beta_{s+1}^{+}(F,G)-\beta_{s+1}^{-}(F,G)}{(s+1)!}(\dirac _1-\dirac)^{s+1}
\\
&\quad+ \bigl[\beta_{s+2}^{+}(F,G)+\beta_{s+2}^{-}(F,G)+
\beta_{s+1}^{-}(F,G) \bigr](\dirac_1-
\dirac)^{s+2}\varTheta C(s).
\end{align*}
\end{lemma}

\begin{proof} Observe that
\begin{align*}
(\dirac_1-\dirac)^{s+1}+(\dirac_{-1}-
\dirac)^{s+1}&= (\dirac_1-\dirac)^{s+1}-(
\dirac_{-1})^{s+1}(\dirac_1-\dirac
)^{s+1}
\\
&=(\dirac_1-\dirac)^{s+1}\dirac_{-1}(
\dirac_1-\dirac)\sum_{j=1}^{s+1}(
\dirac_{-1})^{s+1-j}
\\
&= (\dirac_1-\dirac)^{s+2}\varTheta(s+1).
\end{align*}
The lemma now follows from (\ref{nuF}).
\end{proof}

%
%
%


\begin{lemma} Let $F\in\F_Z$ with mean $\mu(F)$ and variance $\sigma
^2(F)$, both finite.
Then, for all $\vert t\vert\leqslant\pi$,
\begin{align}
\bigl\vert\w F(t)\bigr\vert&\leqslant1-\frac{(1-\| (\dirac_1-\dirac)F\|
/2)t^2}{4\pi
}
\nonumber
\\
&\leqslant\exp \biggl\{-\frac{(1-\| (\dirac_1-\dirac)F\|/2)}{\pi
}\sin^2\frac {t}{2}
\biggr\},\label{mineka2}
\\
\bigl\vert \bigl(\w F(t)\ee^{-\ii t\mu(F)} \bigr)'\bigr\vert&\leqslant
\pi^2\sigma^2(F)\bigl\vert\sin(t/2)\bigr\vert.\label{roos}
\end{align}
\end{lemma}

The first estimate in (\ref{mineka2}) is given in \cite{BKN14} p. 884,
the second estimate in (\ref{mineka2}) is trivial. For the proof of
(\ref{roos}), see p.~81 in \cite{Ce16}.


\begin{lemma}\label{varijotas} Let $M\in\M$ be concentrated on $\ZZ$,
$\sum_{k\in\ZZ}\vert k\vert\vert M\{k\}\vert<\infty$. Then, for any
$a\in\RR$, $b>0$ the following inequality holds
\begin{equation*}
\| M\|\leqslant(1+b\pi)^{1/2} \Biggl( \frac{1}{2\pi}\int
_{-\pi}^\pi \biggl( \bigl\vert\w M(t) \bigr\vert^2+
\frac{1}{b^2} \bigl\vert \bigl(\ee^{-\ii ta}\w M(t) \bigr)'
\bigr\vert^2 \biggr)\,\dd t \Biggr)^{1/2}. \label{var}
\end{equation*}
\end{lemma}
Lemma \ref{varijotas} is a well-known inversion inequality for lattice
distributions. Its proof can be found, for example, in \cite{Ce16},
Lemma 5.1. 


\begin{lemma}\label{MBlem} Let $H_{in}=\eL(w_iS_i)$ and let $D_{in}$ be
defined by (\ref{Din}), $i=1,\dots,N$. Let conditions (\ref{cond1})
hold. Then, for $i=1,2,\dots,N$,
\begin{align*}
H_{in}-D_{in}&=\qubar_i(p_i+
\qubar_i)Y_i\exp\{n_i\gamma_iY_i/60
\}\varTheta C+(p_i+ \qubar_i) (\dirac_{w_i}-
\dirac)\varTheta C \ee^{-C_in_i},
\\
H_{in}&=\exp\{n_i\gamma_iY_i/30
\} \varTheta C+(p_i+\qubar_i) (\dirac_{w_i}-
\dirac)\varTheta C\ee^{-C_in_i},
\\
D_{in}&=\exp\{n_i\gamma_iY_i/30
\} \varTheta C,\qquad\ee^{-C_in_i}\leqslant\frac{C(k_0)\qubar
_i}{\sqrt{\max(n_i\qubar_i,1)}},
\\
\bigl\vert\w Y_i(t)\bigr\vert&\leqslant4\bigl\vert\sin(tw_i/2)\bigr\vert,
\qquad Re \w Y_i(t)\geqslant-\frac{4}{3}\sin^2(tw_i/2),
\qquad\frac{\qubar
_i}{2}\leqslant\gamma_i\leqslant
\qubar_i. 
\end{align*}
\end{lemma}

\begin{proof} The statements follow from Lemma 5.4, Lemma 5.1 and the
relations given on pp.~1131--1132 in \cite{CeVe10}. The estimate for
$\ee^{-C_in_i}$ follows from the first assumption in (\ref{cond1}) and
the following simple estimate
\begin{align*}
\ee^{-C_in_i}&\leqslant\ee^{-C_in_i/2}\ee^{-C_in_i\qubar
_i/2}\leqslant
\frac{C(k_0)}{n_i^{k_0}}\frac{2}{1+C_in_i\qubar_1}
\\
&\leqslant\frac{C(k_0)\qubar_i}{\min(1,C_i)(1+n_i\qubar
_i)}\leqslant\frac{C(k_0)\qubar_i}{\min(1,C_i)\max(n_i\qubar
_i,1)}.\qedhere
\end{align*}
\end{proof}




\section{Proofs for sums of independent rvs}

\begin{proof}[Proof of Theorem \ref{Teorema}] Let $F_{ij,w}$ (resp.
$G_{ij,w}$) denote the distribution of $w_iX_{ij}$ (resp. $w_iZ_{ij}$).
Note that
$\w F_{ij,w}(t)=\w F_{ij}(w_it)$. By the triangle inequality
\begin{align*}
\bigl\vert\eL(S)-\eL(Z)\bigr\vert_K&= \Biggl\vert\prod
_{i=1}^N \eL (w_iS_i)-
\prod_{i=1}^N \eL(w_iZ_i)
\Biggr\vert_K
\\
&\leqslant\sum_{i=1}^N \Biggl\vert \bigl(
\eL(w_iS_i)-\eL (w_iZ_i)
\bigr) \prod_{l=1}^{i-1}\eL(w_lS_l)
\prod_{l=i+1}^N\eL (w_lZ_l)
\Biggr\vert_K.
\end{align*}
Similarly,
\begin{align*}
\eL(w_iS_i)-\eL(w_iZ_i)&=
\prod_{j=1}^{n_i}F_{ij,w}-\prod
_{j=1}^{n_i}G_{ij,w}
\\
&= \sum_{j=1}^{n_i}(F_{ij,w}-G_{ij,w})
\prod_{k=1}^{j-1}F_{ik,w}\prod
_{k=j+1}^{n_i}G_{ik,w}.
\end{align*}
For the sake of brevity, let
\begin{align*}
E_{ij}&:=\prod_{k=1}^{j-1}F_{ik,w}
\prod_{k=j+1}^{n_i}G_{ik,w},
\\
T_i&:=\prod_{l=1}^{i-1}
\eL(w_lS_l)\prod_{l=i+1}^N
\eL(w_lZ_l)= \prod_{l=1}^{i-1}
\prod_{m=1}^{n_l}F_{lm,w} \prod
_{l=i+1}^N\prod_{m=1}^{n_l}G_{lm,w}.
\end{align*}
Then, combining both equations given above with Lemma \ref{f-g} , we get
\begin{align}
\bigl\vert\eL(S)-\eL(Z)\bigr\vert_K&\leqslant C(s) \sum
_{i=1}^N\sum_{j=1}^{n_i}
\bigl[\beta_{s+1}^+(F_{ij},G_{ij})
\nonumber
\\
&\quad+ \beta_{s+1}^{-} (F_{ij},G_{ij})
\bigr]\bigl\vert(\dirac_{w_i}-\dirac)^{s+1}E_{ij}T_i
\bigr\vert_K.\label{bndr1}
\end{align}

Let $\vert t\vert\leqslant\pi/\max_iw_i$. Then it follows from
(\ref
{mineka2}) that
\begin{equation}
\bigl\vert\w E_{ij}(t)\w T_i(t)\bigr\vert\leqslant
\ee^{u_{ij}\sin
^2(tw_i/2)/\pi
} \exp \Biggl\{-\frac{1}{\pi}\sum
_{l=1}^N \sum_{m=1}^{n_l}u_{lm}
\sin ^2\frac {tw_l}{2} \Biggr\}. \label{expo}
\end{equation}
Observe that $\ee^{u_{ij}\sin^2(tw_i/2)/\pi}\leqslant\ee^{1/\pi}=C$.
Next, let
\begin{equation}
L:=\frac{1}{8\pi}\sum_{l=1}^N\sum
_{m=1}^{n_l}u_{lm} \bigl[(
\dirac_{w_l}-\dirac)+(\dirac_{-w_l}-\dirac) \bigr].
\label{virsus}
\end{equation}
It is not difficult to check, that $\exp\{L\}$ is a CP distribution
with non-negative characteristic function. Also, by the definition of
exponential measure, $\exp\{-L\}$, which can be called \emph{the
inverse} to $\exp\{L\}$, is a signed measure with finite variation.
We have
\begin{equation}
\vert(\dirac_{w_i}-\dirac)^{s+1}E_{ij}T_i
\vert_K= \vert(\dirac_{w_i}-\dirac)^{s+1}E_{ij}T_i
\exp\{-L\}\exp\{L\}\vert _K. \label{jo}\vadjust{\goodbreak}
\end{equation}
Next step is similar to the definition of \emph{mod}-Poisson
convergence. We apply Lemma \ref{ad} with $h=\max w_i/\pi$ and
$U_1=\exp\{L\}$ and $W_1=(\dirac_{w_i}-\dirac
)^{s+1}E_{ij}T_i\exp\{-L\}$. By Lemma \ref{ac},
\begin{align}
Q\bigl(\exp\{L\},h\bigr)&\leqslant C\frac{\max w_i}{\min w_i}\cdot Q\bigl(\exp\{L\} ,
\min w_i/2\bigr)
\nonumber
\\
&\leqslant C\frac{\max w_i}{\min w_i} \Biggl( \sum_{l=1}^N
\sum_{m=1}^{n_l}u_{lm}
\Biggr)^{-1/2}. \label{quu}
\end{align}
From (\ref{expo}) and (\ref{virsus}), it follows that
\begin{align}
\biggl\vert\frac{\w W_1(t)}{t} \biggr\vert\cdot\frac{1}{h}&\leqslant C(s)
\frac{\vert
\sin(tw_i/2)\vert^{s+1}}{h\vert t\vert} \exp \Biggl\{-\frac{1}{2\pi}\sum
_{l=1}^N \sum_{m=1}^{n_l}u_{lm}
\sin^2\frac {tw_l}{2} \Biggr\}
\nonumber
\\
&\leqslant C(s)\frac{w_i}{h}\bigl\vert\sin(tw_i/2)
\bigr\vert^s \exp \Biggl\{-\frac{1}{2\pi }\sum
_{m=1}^{n_i}u_{im} \sin ^2(tw_i/2)
\Biggr\}
\nonumber
\\
&\leqslant C(s) \Biggl(\sum_{m=1}^{n_i}u_{im}
\Biggr)^{-s/2}. \label{jo2}
\end{align}
It remains to estimate $\| W_1\|$.
Let
\begin{align*}
\varPhi_{lm,w}&:=F_{lm,w}\exp \biggl\{\frac{1}{8\pi }u_{lm}
\bigl[(\dirac_{w_l}-\dirac)+(\dirac_{-w_l}-\dirac) \bigr]
\biggr\},
\\
\varPsi_{lm,w}&:=G_{lm,w}\exp \biggl\{\frac{1}{8\pi}u_{lm}
\bigl[(\dirac_{w_l}-\dirac)+(\dirac_{-w_l}-\dirac) \bigr]
\biggr\}
\end{align*}
Then by the properties of the total variation norm,
\begin{align}
\| W_1\|&\leqslant \biggl\|\exp \biggl\{\frac{1}{8}u_{ij}
\bigl[(\dirac_{w_i}-\dirac)+(\dirac_{-w_i}-\dirac) \bigr]
\biggr\} \biggr\|
\nonumber
\\
&\quad\times \Biggl\|(\dirac_{w_i}-\dirac)^{s+1}\prod
_{k=1}^{j-1}\varPhi_{ik,w}\prod
_{k=j+1}^{n_i}\varPsi_{ik,w} \Biggr\|
\nonumber
\\
&\quad\times\prod_{l=1}^{i-1} \Biggl\|\prod
_{m=1}^{n_l}\varPhi _{lm,w} \Biggr\|\prod
_{l=i+1}^N \Biggl\|\prod
_{m=1}^{n_l}\varPsi_{lm,w} \Biggr\|. \label{W1a}
\end{align}

The first norm in (\ref{W1a}) is bounded by $\exp \{\frac
{1}{8}u_{ij}[\| \dirac_{w_i}-\dirac\|+\| \dirac_{-w_i}-\dirac\|
] \}\leqslant
\exp\{1/2\}$. The total variation norm is invariant with respect to
scale. Therefore, without loss of generality, we can switch to $w_l=1$.
In this case, we use the notations $\varPhi_{ik},\varPsi_{ik}$. Then, again
employing the inverse CP measures, we get
\begin{align*}
& \Biggl\|(\dirac_{w_i}-\dirac)^{s+1}\prod
_{k=1}^{j-1}\varPhi _{ik,w}\prod
_{k=j+1}^{n_i}\varPsi_{ik,w} \Biggr\|
\\
&\quad= \Biggl\|(\dirac_{1}-\dirac)^{s+1}\prod
_{k=1}^{j-1}\varPhi _{ik}\prod
_{k=j+1}^{n_i}\varPsi_{ik} \Biggr\|
\\
&\quad= \Biggl\|(\dirac_{1}-\dirac)^{s+1}\prod
_{k=1}^{j-1}\varPhi _{ik}\prod
_{k=j+1}^{n_i}\varPsi_{ik} \exp\bigl
\{u_{ij}(\dirac_1-\dirac )\bigr\}\exp\bigl
\{u_{ij}(\dirac-\dirac_1)\bigr\} \Biggr\|
\\
&\quad\leqslant\ee^2 \Biggl\|(\dirac_{1}-\dirac)^{s+1}
\exp\bigl\{ u_{ij}(\dirac_1-\dirac)\bigr\}\prod
_{k=1}^{j-1}\varPhi_{ik}\prod
_{k=j+1}^{n_i}\varPsi_{ik} \Biggr\|.
\end{align*}
We apply Lemma \ref{varijotas} with $a=u_{ij}+\sum_{k\ne
i}^{n_i}\mu_{ik}$, $b=1$, where $\mu_{ik}=\nu_1^{+}(
F_{ik})+\nu_1^{-}(F_{ik})$ is the mean of $F_{ik}$ and, due to
assumption (\ref{sal2}), of $G_{ik}$. Let
\[
\w\Delta(t):= \bigl(\ee^{\ii t}-1 \bigr)^{s+1} \exp\bigl
\{u_{ij} \bigl(\ee^{\ii t}-1-it \bigr)\bigr\}\prod
_{k=1}^{j-1}\w\varPhi_{ik}(t)
\ee^{-\ii t\mu_{ik}}\prod_{k=j+1}^{n_i}\w
\varPsi_{ik}\ee^{-\ii t\mu_{ik}}.
\]
It follows from (\ref{mineka2}) that
\begin{align*}
\bigl\vert\Delta(t)\bigr\vert&\leqslant C(s)\bigl\vert\sin(t/2)\bigr\vert ^{s+1}\exp \Biggl
\{- \frac {1}{2\pi}\sum_{m=1}^{n_i}u_{im}
\sin ^2(t/2) \Biggr\}
\\
&\leqslant C(s) \Biggl(\sum_{m=1}^{n_i}u_{im}
\Biggr)^{-s/2}.
\label{fin1}
\end{align*}
For the estimation of $\vert\Delta'(t)\vert$, observe that by (\ref
{mineka2}) and (\ref{roos})
\begin{align*}
\bigl\vert \bigl(\w\varPhi_{ik}(t)\ee^{-\ii t\mu_{ik}} \bigr)'
\bigr\vert& \leqslant \biggl\vert\w F_{ik}(t)\ee^{-\ii t\mu_{ik}}\frac
{u_{ik}}{\pi}
\sin(t/2)\ee^{(u_{ik}/2\pi)\sin^2(t/2)} \biggr\vert
\\
&\quad+ \bigl\vert \bigl(\w F_{ik}(t)\ee^{-\ii t\mu_{ik}}
\bigr)' \ee^{(u_{ik}/2\pi )\sin^2(t/2)} \bigr\vert
\\
&\leqslant C(s) \bigl(u_{ik}+\sigma^2_{ik}
\bigr)\bigl\vert\sin(t/2)\bigr\vert
\\
&\leqslant C(s) \bigl(u_{ik}+\sigma^2_{ik}
\bigr)\bigl\vert\sin(t/2)\bigr\vert\exp \biggl\{-\frac {u_{ik}}{\pi}\sin ^2(t/2)
\biggr\} \ee^{1/\pi}.
\end{align*}
The same bound holds for $\vert(\w\varPsi_{ik}(t)\exp\{-\ii t\mu
_{ik}\})'\vert$. The direct calculation shows that
\[
\bigl\vert \bigl( \bigl(\ee^{\ii t}-1 \bigr)^{s+1}\exp\bigl
\{u_{ij} \bigl(\ee ^{\ii t}-1-\ii t \bigr)\bigr\}
\bigr)'\bigr\vert\leqslant
C(s)
\bigl\vert\sin(t/2)\bigr\vert^{s}\exp \biggl\{-\frac{1}{\pi}
u_{ij} \sin ^2(t/2) \biggr\}.
\]
Taking into account of the previous two estimates, it is not difficult
to prove that
\begin{align*}
\big\vert\Delta'(t)\big\vert&\leqslant C(s)\big\vert\sin(t/2)
\big\vert^s \exp \Biggl\{-\frac{1}{\pi }\sum
_{k=1}^{n_i}u_{ik} \sin^2(t/2)
\Biggr\}
\\
&\quad\times \Biggl(1+\sin^2(t/2)\sum_{k=1,k\ne j}^{n_i}
\bigl(u_{ik}+\sigma^2_{ik} \bigr) \Biggr)
\nonumber
\\
&\leqslant C(s) \Biggl(\sum_{k=1}^{n_i}u_{ik}
\Biggr)^{-s/2} \Biggl(1+\sum_{k=1}^{n_i}
\sigma^2_{ik} / \sum_{k=1}^{n_i}u_{ik}
\Biggr). 
\end{align*}
From Lemma \ref{varijotas}, it follows that
\begin{align}
\Biggl\|(\dirac_{w_i}-\dirac)^{s+1}\prod
_{k=1}^{j-1}\varPhi _{ik,w}\prod
_{k=j+1}^{n_i}\varPsi_{ik,w} \Biggr\| \leqslant C(s)
\Biggl(\sum_{k=1}^{n_i}u_{ik}
\Biggr)^{-s/2} \Biggl(1+\sum_{k=1}^{n_i}
\sigma^2_{ik} / \sum_{k=1}^{n_i}u_{ik}
\Biggr).\label{fin3}
\end{align}
The remaining two norms in (\ref{W1a}) can be estimated similarly:
\begin{equation}
\Biggl\|\prod_{m=1}^{n_l}\varPhi_{lm,w}
\Biggr\|, \Biggl\|\prod_{m=1}^{n_l}\varPsi _{lm,w}
\Biggr\|\leqslant C \Biggl(1+\sum_{m=1}^{n_l}
\sigma^2_{lm} /\sum_{m=1}^{n_l}u_{lm}
\Biggr). \label{fin4}
\end{equation}
Substituting (\ref{fin3}), (\ref{fin4}) into (\ref{W1a}), we obtain
\begin{equation}
\| W_1\|\leqslant C(N,s) \Biggl(\sum_{m=1}^{n_i}u_{im}
\Biggr)^{-s/2} \prod_{l=1}^N
\Biggl(1+\sum_{k=1}^{n_l}\sigma^2_{lk}
/\sum_{k=1}^{n_l} u_{lk}
\Biggr).\label{fin5}
\end{equation}
Combining (\ref{fin5}) with (\ref{quu}), (\ref{jo2}) and (\ref
{jo}), we get
\begin{align*}
\bigl\vert(\dirac_{w_i}-\dirac)^{s+1}E_{ij}T_i
\bigr\vert_K &\leqslant C(N,s)\frac{\max_jw_j}{\min_j
w_j} \Biggl(\sum
_{i=1}^N \sum_{k=1}^{n_i}u_{ik}
\Biggr)^{-1/2}
\\
&\quad\times \Biggl(\sum_{m=1}^{n_i}u_{im}
\Biggr)^{-s/2}\prod_{l=1}^N
\Biggl(1+\sum_{k=1}^{n_l}\sigma^2_{lk}
/\sum_{k=1}^{n_l} u_{lk} \Biggr).
\end{align*}
Substituting the last estimate into (\ref{bndr1}) we complete the
proof of
(\ref{BTa}).
The proof of (\ref{BTb}) is very similar and, therefore, omitted.
\end{proof}


\begin{proof}[Proof of Theorem \ref{TeoremaID}] We outline only the
differences from the proof of Theorem \ref{Teorema}.
No use
of convolution with the inverse Poisson measure is required, since we
have powers of $F_i^{n_i}$, which can be used for Levy's
concentration function. Let $\lfloor a\rfloor$
denote an integer part of $a$ and let $a(k):=\lfloor(k-1)/2\rfloor$,
$b(k):=\lfloor(n_i-k)/2\rfloor$. Then, as in the proof of Theorem
\ref
{Teorema}, we obtain
\begin{align*}
\bigl\vert\eL(S)-\eL(Z)\bigr\vert_K&\leqslant C(s)\sum
_{i=1}^N\sum_{k=1}^{n_i}
\bigl(\beta_{s+1}^{+}(F_i,G_i)+
\beta_{s+1}^{-}(F_i,G_i) \bigr)
\\
&\quad\times \Biggl\vert(\dirac_{w_i}-\dirac)^{s+1}F_{iw}^{a(k)}G_{iw}^{b(k)}
F_{iw}^{a(k)}G_{iw}^{b(k)}\prod
_{j=1}^{i-1}F_{jw}^{n_j} \prod
_{j=i+1}^NG_{jw}^{n_j}
\Biggr\vert_K.
\end{align*}
Here $F_{iw}$ and $G_{iw}$ denote the distributions of $w_iX_{ij}$ and
$w_iZ_{ij}$, respectively. We can apply Lemma \ref{ad} to the
Kolmogorov norm given above, taking $W=(\dirac_{w_i}-\dirac
)^{s+1}F_{iw}^{a(k)}G_{iw}^{b(k)}$. The remaining distribution is used
in Levy's
concentration function. The Fourier--Stieltjes transform $\w W(t)/t$ is
estimated exactly as in the proof of Theorem \ref{Teorema}. The total
variation of any distribution is equal to 1, therefore $\| W\|\leqslant
\| \dirac_{w_i}-\dirac\|\leqslant2$ and we can avoid
application of Lemma \ref{varijotas}.
\end{proof}


\begin{proof}[Proof of Corollary \ref{cornb}] As proved in \cite{BaXi99}, p.~144,
\[
\frac{1}{2}\bigl\| G_k(\dirac_1-\dirac)\bigr\|\leqslant
\biggl(\frac{p_k\nu
_1^{+}(F_k)}{q_k}\ln\frac{1}{p_k} \biggr)^{-1/2}.
\]
Observe that $
\nu_1^{+}(F_j)=\nu_1^{+}(G_j)$ and $\nu_2^{+}(F_j)=\nu_2^{+}(G_j)$. It
remains to find $\nu_3^+(G_j)$ and apply Theorem \ref{TeoremaID}.
\end{proof}
\section{Proof of Theorem \ref{TMARK}}

The proof is similar to the one given in \cite{SlCe16}. Let
$A_i=\exp\{n_i\gamma_iY_i/30\}$.
From Lemma \ref{MBlem}, it follows that
\[
H_{in}=A_i\varTheta_i C +\ee^{-C_in_i}
\varTheta_i C,\qquad D_{in}=A_i
\varTheta_i C,\quad i=1,2,\dots,N.
\]
Here we have added index to $\varTheta_i$ emphasizing that they might be
different for different $i$.
As usual, we assume that the convolution $\prod_{k=N+1}^N=\prod_{k=1}^0=\dirac$.
Let also denote by $\sum_i^*$ summation over all indices
$\{j_1,j_2,\dots,j_{i-1}\in\{0,1\}\}$.
Taking into account Lemma \ref{MBlem} and the properties of the
Kolmogorov and total variation norms given in the Introduction, we get
\begin{align}
& \Biggl\vert\prod_{i=1}^NH_{in}-
\prod_{i=1}^ND_{in} \Biggr\vert
_K
\nonumber
\\
&\quad\leqslant\sum_{i=1}^N
\Biggl\vert(H_{in}-D_{in})\prod_{k=1}^{i-1}H_{kn}
\prod_{k=i+1}^ND_{kn}
\Biggr\vert_K
\nonumber
\\
&\quad\leqslant\sum_{i=1}^N
\Biggl\vert(H_{in}-D_{in})\sum\nolimits ^*_i
\prod_{k=1}^{i-1}A_k^{j_k}
\varTheta_kC
\nonumber
\\
&\qquad \times\prod_{k=i+1}^NA_k
\varTheta_k C \prod_{k=1}^{i-1}\ee
^{-(1-j_k)n_kC_k}\varTheta_kC \Biggr\vert_K
\nonumber
\\
&\quad\leqslant C(N)\sum_{i=1}^N
\qubar_i(p_i+\qubar_i)\sum
\nolimits_i^{*} \Biggl\vert Y_i\exp
\{n_i\gamma_iY_i/60\} \prod
_{k=1}^{i-1}A_k^{j_k}\prod
_{k=i+1}^NA_k
\Biggr\vert_K
\nonumber
\\
&\qquad\times\prod_{k=1}^{i-1}
\ee^{-(1-j_k)n_kC_k} +C\sum_{i=1}^N(p_i+
\qubar_i)\ee^{-C_in_i}
\nonumber
\\
&\qquad\times\sum\nolimits_i^{*} \Biggl\vert(
\dirac_{w_i}-\dirac)\prod_{k=1}^{i-1}A_k^{j_k}
\prod_{k=i+1}^NA_k
\Biggr\vert_K\prod_{k=1}^{i-1}
\ee^{-(1-j_k)n_kC_k}.\label{marko1}
\end{align}

Both summands on the right-hand side of (\ref{marko1}) are estimated
similarly. Observe that
\begin{align*}
& \Biggl\vert Y_i\exp\{n_i\gamma_iY_i/60
\} {\prod_{k=1}^{i-1}A_k^{j_k}
\prod_{k=i+1}^NA_k}_K
\Biggr\vert
\\
&\quad= \Biggl\vert Y_i \exp \Biggl\{\frac{n_i\gamma_iY_i}{60}+
\frac
{1}{30} \sum_{k=1}^{i-1}j_kn_k
\gamma_kY_k+\frac{1}{30}\sum
_{k=i+1}^Nn_k\gamma_kY_k
\Biggr\} \Biggr\vert_K.
\end{align*}
Next we apply Lemma \ref{ad} with $W=Y_i$ and $h=\max w_i/\pi$ and
$V$ with
\begin{align*}
\w V(t)&= \exp \Biggl\{-\frac{1}{90} \Biggl[\sum
_{k=1}^{i-1}j_k\max(n_k
\qubar_k,1)\sin^2(tw_k/2)
\\
&\quad+\sum_{k=i}^N\max(n_k
\qubar_k,1)\sin^2(tw_k/2) \Biggr] \Biggr\}.
\end{align*}
By Lemma \ref{MBlem}
\[
\frac{\vert\w Y_i(t)\vert}{t}\frac{1}{h}+\| Y_i\|\leqslant C.
\]
Observe that
\begin{align*}
& \Biggl\vert\exp \Biggl\{\frac{n_i\gamma_i}{60}\w Y_i(t)+\frac
{1}{30}
\sum_{k=1}^{i-1}j_kn_k
\gamma_k\w Y_k(t)+\frac{1}{30}\sum
_{k=i+1}^N\gamma _k\w Y_k(t)
\Biggr\} \Biggr\vert
\nonumber
\\
&\quad\leqslant \exp\Biggl\{-\frac{n_i\gamma_i\sin^2(tw_i/2)}{45}-\frac{2}{45}\sum
_{k=1}^{i-1}j_kn_k
\gamma_k\sin^2(tw_k/2)
\\
&\qquad-\frac{2}{45}\sum_{k=i+1}^N
n_k\gamma_k\sin^2(tw_k/2) \Biggr
\}
\nonumber
\\
&\quad\leqslant\exp \Biggl\{-\frac{1}{90} \Biggl[\sum
_{k=1}^{i-1}j_kn_k\qubar
_k\sin^2(tw_k/2) +\sum
_{k=i}^Nn_k\qubar _k
\sin^2(tw_k/2) \Biggr] \Biggr\}
\nonumber
\\
&\quad\leqslant\ee^{N/90}\exp\Biggl\{-\frac{1}{90} \Biggl[\sum
_{k=1}^{i-1}j_k(n_k
\qubar_k+1)\sin^2(tw_k/2) \\
&\qquad+\sum
_{k=i}^N(n_k\qubar_k+1)
\sin^2(tw_k/2) \Biggr] \Biggr\}
\nonumber
\\
&\quad\leqslant\ee^{N/90} \exp\Biggl\{-\frac{1}{90} \Biggl[\sum
_{k=1}^{i-1}j_k
\max(n_k\qubar_k,1)\sin^2(tw_k/2)
\\
&\qquad+\sum_{k=i}^N\max(n_k
\qubar_k,1)\sin^2(tw_k/2) \Biggr] \Biggr\}
\nonumber
\\
&\quad=\ee^{N/90}\w V(t).
\end{align*}
Therefore, using Lemma \ref{ac}, we prove
\begin{align}
& \Biggl\vert Y_i\exp\{n_i\gamma_iY_i/60
\} \prod_{k=1}^{i-1}A_k^{j_k}
\prod_{k=i+1}^NA_k
\Biggr\vert_K\notag
\\
&\quad\leqslant C(N)Q\Bigl(V,\max_i w_i/h
\Bigr)
\nonumber
\\
&\quad\leqslant C(N) \biggl(\frac{\max w_i}{\min w_i} \biggr) Q(V,\min
w_i/2)
\nonumber
\\
&\quad\leqslant C(N) \biggl(\frac{\max w_i}{\min w_i} \biggr) \Biggl(\sum
_{k=1}^{i-1}j_k\max(n_k
\qubar_k,1)+\sum_{k=i+1}^N
\max(n_k\qubar _k,1) \Biggr)^{-1/2}.\label{mmm}
\end{align}
Next observe that by Lemma \ref{MBlem},
\begin{align*}
\Biggl\vert\prod_{k=1}^{i-1}\ee^{-(1-j_k)n_kC_k}
\Biggr\vert &= C\exp \Biggl\{-\sum_{k=1}^{i-1}(1-j_k)C_kn_k
\Biggr\}
\\
&\leqslant\frac{C(k_0,N)}{\max(1,\sqrt{\sum_{k=1}^{i-1}(1-j_k)\max
(n_k\qubar
_k,1)} )}.
\end{align*}
The last estimate, (\ref{mmm}) and the trivial inequality
$1/(ab)<2/(a+b)$, valid for any $a,b\geqslant1$, allow us to obtain
\begin{align*}
&\sum_{i=1}^N\qubar_i(p_i+
\qubar_i)\sum\nolimits_i^{*} \Biggl\vert
Y_i\exp\{n_i\gamma_iY_i/60\}
\prod_{k=1}^{i-1}A_k^{j_k}
\prod_{k=i+1}^NA_k
\Biggr\vert_K \prod_{k=1}^{i-1}
\ee^{-(1-j_k)n_kC_k}
\\
&\quad\leqslant C(k_0,N)\frac{\max w_j}{\min w_j}\cdot \frac{\sum_{i=1}^N\qubar_i(p_i+\qubar_i)}{\sqrt{\sum_{k=1}^N\max
(n_k\qubar_k,1)}}.
\end{align*}
The estimation of the second sum in (\ref{marko1}) is almost identical
and, therefore, omitted. \qed


\begin{acknowledgement}
The main part of the work was accomplished during the first
author's stay at the Department of Mathematics, IIT Bombay, during
January, 2018. The first author would like to thank the members
of the Department for their hospitality.
We are grateful to the
referees for useful remarks.
\end{acknowledgement}


%
\end{document}